\documentclass[12pt,centertags]{amsart}
\usepackage{amssymb}
\usepackage[all]{xy}
\usepackage{amsmath,amstext,amsthm,a4,amssymb,amscd}
\usepackage[mathscr]{eucal}
\usepackage{mathrsfs}
\usepackage{epsf}
\usepackage{tikz}
\usepackage{extarrows}
\usetikzlibrary{matrix,arrows,decorations.pathmorphing}
\usepackage[a4paper,top=3cm,bottom=3cm,
left=2.2cm,right=2.2cm]
{geometry}
\numberwithin{equation}{section}
\allowdisplaybreaks[3]

\usepackage{color}
\usepackage{soul}

\newcommand{\field}[1]{\mathbb{#1}}
\newcommand{\C}{\field{C}}
\newcommand{\N}{\field{N}}

\newcommand{\R}{\field{R}}

\def\Im{{\rm Im}}

\DeclareMathOperator{\ch}{ch}

\newtheorem{thm}{Theorem}[section]
\newtheorem{lemma}[thm]{Lemma}
\newtheorem{prop}[thm]{Proposition}

\theoremstyle{definition}

\theoremstyle{definition}
\newtheorem{defn}[thm]{Definition}

\newcommand{\be}{\begin{eqnarray}}
	\newcommand{\ee}{\end{eqnarray}}

\numberwithin{equation}{section}
\numberwithin{thm}{section}

\newcommand{\comment}[1]{}

\begin{document}
	
	\title[$\lambda$-ring structure in differential K-theory]
	{$\lambda$-ring structure in differential K-theory}
	
	\author{Bo LIU}
	\address{School of Mathematical Sciences,  Key Laboratory of 
		MEA (Ministry of Education) \& Shanghai Key Laboratory 
		of PMMP,  East China Normal University, Shanghai 
		200241, China}
	\email{bliu@math.ecnu.edu.cn}
	
	\author{Xiaonan MA}
	
	\address{Chern Institute of Mathematics \& LPMC, Nankai 
		University, Tianjin 300071, P.R. China 
	}  
	\email{xiaonan.ma@nankai.edu.cn}	
	
	\date{\today}
	
	\begin{abstract}
		We establish the splitting principle for differential K-theory, a refinement of topological K-theory that incorporates geometric data via differential forms. Using this principle, we prove that the differential $K^0$-ring associated to closed smooth manifolds admits a $\lambda$-ring structure. This structure enables a concrete construction of the Adams operations in differential K-theory introduced by Bunke. At last, we extend all these results to an equivariant setting associated with a compact Lie group action.
	\end{abstract}
	
	\maketitle
	
	
	\setcounter{section}{-1}
	

	\section{Introduction}\label{s00}  
	
	
	The notion of 
	$\lambda$-rings was first introduced by Grothendieck 
	\cite{Grothendieck} in 1956 and later studied by Atiyah 
	and Tall \cite{ATall69}; see also \cite{BerthelotTh}. 
	
	For a commutative ring $R$ with identity, a 
	$\lambda$-structure on $R$ consists of a countable family of
	maps $\lambda^n:R\to R$, $n\in \N$, satisfying 
	some conditions (see Definition \ref{defn:1.06}).  
	If only a subset of these conditions is satisfied, we refer 
	to it as a pre-$\lambda$-ring structure
	(Definition \ref{defn:2.01}). The $\lambda$-ring structure is 
	closely related to Adams operations: if $R$ is a $\lambda$-ring, 
	one can associate naturally
	Adams operations. 
	Conversely, for a pre-$\lambda$-ring $R$, with compatible  
	Adams operations, if $R$ is torsion free, 
	the pre-$\lambda$-ring
	structure of $R$ is a $\lambda$-ring structure (Theorem
	\ref{thm:1.07}). 
	
	It is well known that, although the topological K-ring
	is not torsion free, it carries a natural
	$\lambda$-ring structure \cite{ATall69}. 
	The corresponding Adams operations plays a key role
	in the proof of Adams' celebrated theorem on the Hopf 
	invariant \cite{A62}.

	In \cite{GS90c}, Gillet and Soul\'e introduced the arithmetic 
	K-theory in  Arakelov geometry. This theory extends 
	Grothendieck's 
	K-theory by adding conjugation-invariant Hermitian metrics 
	on homomorphic vector bundles over the complex points and 
	differential forms of type $(p,p)$ modulo
	$\mathrm{Im}\partial+\mathrm{Im}\bar{\partial}$.
	A pre-$\lambda$-ring structure on the arithmetic K-ring 
	was constructed in \cite{GS90c}, which was later shown 
	by Roessler \cite{Ro01} to be a $\lambda$-ring structure. 
	
	Differential K-theory, introduced by Freed-Hopkins
	\cite{FH00} and developed further by Hopkins-Singer \cite{HS05},
	Simons-Sullivan \cite{SS10}, Bunke-Schick \cite{BS09}, Freed-Lott \cite{FreedLott10}, etc., provides 
	a smooth analogue of arithmetic K-theory.
	It extends topological K-theory by adding Hermitian 
	metrics and connections on complex vector bundles and differential
	forms modulo exact forms.
	In \cite{Bunke10}, Bunke defined Adams operations on
	the differential K-ring directly in a sophisticated topological way and obtained
	an associated Adams-Riemann-Roch theorem.
	In \cite{LM20}, the authors explicitly constructed 
	a pre-$\lambda$-ring structure on the differential K-ring, which plays important role in their work
	on the localization formula for $\eta$-invariants. 
	In this paper, using the splitting principle for differential
	K-theory, which is established in Theorem 
	\ref{prop:2.2},  we prove in Theorem \ref{thm:1.09} that this
	pre-$\lambda$-ring structure is in fact a 
	$\lambda$-ring structure,  and we show that the corresponding 
	Adams operations yield a constructive realization
	of those defined in \cite{Bunke10}.
	
	This paper is organized as follows. In Section 
	\ref{s01}, we review the basic facts on $\lambda$-ring.
	In Section \ref{s02}, we 
	prove the splitting principle for differential K-group
	and state our main result Theorem \ref{thm:1.09}
	that any differential K-ring
	has a natural $\lambda$-ring structure.
	In Section \ref{s03}, we prove Theorem 
	\ref{thm:1.09}.
	In Section
	\ref{s04}, we give the explicit construction of the 
	associated Adams operations in differential K-theory.
	In Section \ref{s05}, we discuss the $\lambda$-ring 
	structure on an equivariant version of differential K-theory.
	
	\
	
	\noindent{\bf Acknowledgments}. 
	BL is partially 
	supported by NSFC No.12225105, National Key R\&D Program of China
	No.2024YFA1013203 and Science and Technology 
	Commission of Shanghai Municipality No.22DZ2229014.
	XM is partially supported by Nankai Zhide Foundation,
	and funded through 
	the Institutional Strategy of the University of Cologne 
	within the German Excellence Initiative.

	
	\section{$\lambda$-ring structure}\label{s01}

	\begin{defn}\label{defn:2.01}\cite[\S V Definition 2.1]{BerthelotTh}
		For a commutative ring $R$ with identity,
		a pre-$\lambda$-ring structure 
		is defined by a countable set of maps $\lambda^n:R\rightarrow R$
		with $n\in \N$ 	such that for all $x,y\in R$,
		\begin{align}\label{eq:1.01}
			\begin{split}
				\lambda^0(x)&=1;  \\
				\lambda^1(x)&=x; \\
				\lambda^n(x+y)&=\sum_{j=0}^n\lambda^j(x)\lambda^{n-j}(y). 
			\end{split}
		\end{align}
		If $R$ has a pre-$\lambda$-ring structure, 
		we call it a pre-$\lambda$-ring.
	\end{defn}
	
	Remark that in \cite[\S I.1]{ATall69}  
	the pre-$\lambda$-ring here is called the $\lambda$-ring.
	
	Let $R$ be a pre-$\lambda$-ring.
	If $t$ is an indeterminate,  we define for $x\in R$,
	\begin{align}\label{eq:2.001}
		\lambda_t(x)=\sum_{n\geq 0}\lambda^n(x)t^n.
	\end{align}
	Then the relations (\ref{eq:1.01}) show that $\lambda_t$ is 
	a homomorphism from 
	the additive group of $R$ into the multiplicative group 
	$1+R[[t]]^+$, of 
	formal power series in $t$ with constant term 1, i.e.,
	\begin{align}\label{eq:2.002}
		\lambda_t(x+y)=\lambda_t(x)\lambda_t(y)\quad 
		\text{for any}\ x,y\in R.
	\end{align} 
	From (\ref{eq:1.01}) and (\ref{eq:2.001}), we have
	\begin{align}\label{eq:1.05b}
		\lambda_t(0)=1.
	\end{align}
	
	Let $\xi_1,\cdots, \xi_q, \zeta_1\cdots, \zeta_r$ be 
	indeterminates and let $s_i$ and $\sigma_i$ be the $i$-th 
	elementary symmetric functions in $\xi_1,\cdots, \xi_q$ and 
	$\zeta_1\cdots, \zeta_r$ respectively. Let $P_n(s_1,\cdots,s_n; 
	\sigma_1,\cdots, \sigma_n)$ be the coefficient of $t^n$ in 
	$\prod_{ij}(1+\xi_i\zeta_jt)$. Let $P_{n,m}(s_1,\cdots, s_{mn})$
	be the coefficient of $t^n$ in
	$$\displaystyle \prod_{i_1<\cdots < i_m}(1+\xi_{i_1}\cdots 
	\xi_{i_m}\,t).$$ We may assume that $r, q\geq n$ for $P_n$ and 
	$q\geq nm$ for $P_{n,m}$. Easy to see that $P_n$ is a polynomial 
	with weight $n$ in $\{s_i \}$ and also weight $n$ in 
	$\{\sigma_i \}$, $P_{n,m}$ is a polynomial with weight $nm$ in 
	$\{s_i \}$. 
	Both $P_n$ and $P_{n,m}$ have integer coefficients and
	so may be defined in any commutative ring.
	
	Let $A$ be any commutative ring with identity.
	By \cite[\S I Lemma 1.1]{ATall69}, $1+A[[t]]^+$ is a 
	pre-$\lambda$-ring with new operations $(\widetilde{+},
	\widetilde{\times}, \tilde{\lambda}^n)$ (with
	$1$ is the ``zero" and $1+t$ is the ``identity")
	defined as follows: for $a_i, b_j\in A$,
	\begin{align}\label{eq:2.041}
		\begin{split}
			\left(1+\sum_{n\geq 1} a_nt^n\right)\widetilde{+}
			\left(1+\sum_{n\geq 1} b_nt^n\right)&:= 
			\left(1+\sum_{n\geq 1} a_nt^n\right)\cdot
			\left(1+\sum_{n\geq 1} b_nt^n\right),\\
			\left(1+\sum_{n\geq 1} a_nt^n\right)\widetilde{\times}\left(
			1+\sum_{n\geq 1} b_nt^n\right)&:= 
			1+\sum_{n\geq 1} P_n(a_1,\cdots,a_n; b_1,\cdots, b_n)\,t^n
		\end{split}
	\end{align}
	and
	\begin{align}\label{eq:1.022}
		\tilde{\lambda}^m\left(1+\sum_{n\geq 1} a_nt^n\right)=1+\sum_{n\geq 1} 
		P_{n,m}(a_1,\cdots, a_{nm})\,
		t^n.
	\end{align}
	In particular, for any $I, J, K\in 1+A[[t]]^+$, we have
	\begin{align}\label{eq:1.05a}
		I\,\widetilde{+} J=J\,\widetilde{+} I,\quad 
		I\,\widetilde{\times} J=J\,\widetilde{\times} I,\quad
		(I\,\widetilde{+} J)\widetilde{\times} K= 
		I\,\widetilde{\times} K \widetilde{+} J\,\widetilde{\times} K,
	\end{align}
	and
	\begin{align}\label{eq:1.06a}
		\tilde{\lambda}^0(I)=1+t,\quad
		\tilde{\lambda}^1(I)=I,\quad 
		\tilde{\lambda}^n(I+J)=\widetilde{+}_{j=0}^n\,
		\left(\tilde{\lambda}^j(I)\widetilde{\times}\tilde{\lambda}^{n-j}
		(J)\right),
	\end{align}
	where $\widetilde{+}_{j=0}^n$ denotes the ``sum" 
	for $j$ from $0$ to $n$ associated with the new ``sum"
	operation
	$\widetilde{+}$.
	
	\begin{defn}\cite[\S I Definition 1.2]{ATall69}\label{defn:1.06}
		We say a pre-$\lambda$-ring structure of $R$ is a 
		$\lambda$-ring structure if $\lambda_t:R\to 1+R[[t]]^+$
		is a homomorphism of pre-$\lambda$-rings,
		i.e., for any $x, y\in R$, $n\in \N$,
		\begin{align}
			\lambda_t(x+y)=\lambda_t(x)\widetilde{+}\lambda_t(y),
			\ \lambda_t(xy)=\lambda_t(x)\widetilde{\times}
			\lambda_t(y),\ 
			\lambda_t(\lambda^n(x))=\widetilde{\lambda}^n(\lambda_t(x)),
		\end{align}
		which reads that
		for any $x, y\in R$,
		\begin{align}\label{eq:1.23}
			\lambda^n(xy)=P_n(\lambda^1(x),\cdots,\lambda^n(x); 
			\lambda^1(y),\cdots, \lambda^n(y)),
			\ \text{i.e.,}\ \lambda_t(xy)=\lambda_t(x)
			\widetilde{\times}
			\lambda_t(y),
		\end{align}
		and
		\begin{align}\label{eq:1.34}
			\lambda^m(\lambda^n(x))=P_{m,n}(\lambda^1(x),
			\cdots, \lambda^{mn}(x)),
			\ \text{i.e.,}\ \lambda_t(\lambda^n(x))
			=\widetilde{\lambda}^n(\lambda_t(x)).
		\end{align}
		If $R$ has a $\lambda$-ring structure, 
		we call it a $\lambda$-ring. 
	\end{defn}
	
	If $R$ is a $\lambda$-ring, from (\ref{eq:1.01}), 
	(\ref{eq:1.06a}) and (\ref{eq:1.34}) for 
	$n=0$, we get
	\begin{align}\label{eq:1.7a}
		\lambda_t(1)=1+t.
	\end{align}
	
	Note that in \cite{ATall69} 
	the $\lambda$-ring here is called the special $\lambda$-ring.
	The following theorem (see \cite[p.49]{LNM308} or 
	\cite[\S V Appendice]{BerthelotTh}) is a convenient 
	criterion to verify whether 
	a pre-$\lambda$-ring is in fact a $\lambda$-ring.
	
	\begin{thm}\label{thm:1.07}
		Let $R$ be a pre-$\lambda$-ring. We assume that $R$ is torsion 
		free, that is, for any $r\in R$, $r\neq 0$ and any 
		$n\in \N^*$, $nr=\underbrace{r+\cdots +r}_n \neq 0$.
		Let $\Psi^n:R\to R$, $n\in \N^*$ be the Adams operations 
		defined by
		\begin{align}\label{eq:1.19}
			\frac{d}{dt}\log\lambda_t(a)=\sum_{n=0}^{+\infty}(-1)^n
			\Psi^{n+1}(a)t^n,\quad 
			\text{for all}\ a\in R.
		\end{align}
		Suppose that for any $n, m\in \N^*$,  $a, b\in R$,
		we have
		\begin{align}\label{eq:1.21}
			\Psi^n(1)=1,\quad
			\Psi^n(ab)=\Psi^n(a)\Psi^n(b),\quad \Psi^n(\Psi^m(a))=\Psi^{nm}(a).
		\end{align}
		Then the pre-$\lambda$-ring structure of $R$ is a $\lambda$-ring 
		structure.
	\end{thm}
	
	
	\section{Differential K-theory}\label{s02}
	
	Let $B$ be a closed manifold.
	Let $\underline{E}=(E, h^E, \nabla^E)$ be a triple which
	consists of a complex vector bundle $E$ over $B$, 
	a Hermitian metric
	$h^E$ on $E$ and a Hermitian connection $\nabla^E$ on $(E, h^E)$.
	As in \cite{LM20}, we call such $\underline{E}$ a geometric
	triple on $B$.
	
	For a geometric triple $\underline{E}$ on $B$,
	denote by $R^E:=(\nabla^E)^2\in \Omega^2(B, \mathrm{End}(E))$ the curvature on $E$.
	The Chern character form (cf. e.g., \cite[(1.22)]{Z01})
	\begin{align}
		\ch(\underline{E}):=\mathrm{Tr}\left[\exp\left(\frac{i}{2\pi}
		R^E\right) \right]\in \Omega^{\mathrm{even}}(B,\R).
	\end{align}
	For two geometric triples $\underline{E_0}$ and 
	$\underline{E_1}$, which have the same underlying
	vector bundle $E$, for the obvious projection
	$\pi:B\times \R\to B$, there exists geometric triple
	$\underline{\pi^*E}$ on $B\times \R$, such that
	\begin{align}\label{eq:2.2}
		\underline{\pi^*E}|_{B\times \{j\}}=\underline{E_j},
		\quad j=0, 1.
	\end{align}
	For $\alpha=\alpha_0+ds\wedge \alpha_1\in \Lambda^{\bullet}
	(T^*(B\times \R))$ with $\alpha_0, \alpha_1
	\in \Lambda^{\bullet}(T^*B)$, we denote by 
	$\{\alpha\}^{ds}:=\alpha_1$.
	Thus the Chern-Simons form 	(cf. e.g., 
	\cite[Definition B.5.3 and Theorem B.5.4]{MM07})
	\begin{align}
		\widetilde{\ch}\left(\underline{E_0}, \underline{E_1}\right):=
		\int_0^1\{\ch(\underline{\pi^*E}) \}^{ds}ds
		\in \Omega^{\mathrm{odd}}(B,\R)/\mathrm{Im}\,d.
	\end{align}
	It only depends on $\nabla^{E_j}$, $j=0,1$,
	and satisfies
	\begin{align}\label{eq:2.04}
		d\,\widetilde{\ch}\left(\underline{E_0}, \underline{E_1}\right)
		=\ch(\underline{E_1})-\ch(\underline{E_0}).
	\end{align}
	
	\begin{defn}\cite[Definition 2.16]{FreedLott10}\label{defn:2.12} 
		A cycle for the 
		differential K-group
		of $B$ is a 
		pair $(\underline{E}, \phi)$ where 
		$\underline{E}=(E, h^E, \nabla^E)$
		is a geometric triple on $B$ and $\phi$
		is an element in $\Omega^{\mathrm{odd}}(B, \R)/\Im \, d$.
		We say two cycles $\left(\underline{E_1}, \phi_1\right)$ 
		and $\left(\underline{E_2}, 
		\phi_2\right)$ are equivalent if there exist a 
		geometric triple $\underline{E_3}=(E_3, h^{E_3}, \nabla^{E_3})$ 
		and a vector bundle isomorphism over $B$
		\begin{align}\label{eq:2.082}
			\Phi:E_1\oplus E_3\rightarrow E_2\oplus E_3
		\end{align}
		such that
		\begin{align}\label{eq:2.083}
			\widetilde{\ch}\left(\underline{E_1}\oplus\underline{E_3}, 
			\Phi^{*} \left(\underline{E_2}\oplus
			\underline{E_3}\right)\right)
			=\phi_2-\phi_1.
		\end{align} 
		We define the sum of the cycles in the natural way and denote by	
		$\widehat{K}_{}^0(B)$  
		the Grothendieck group of equivalence classes of cycles.
	\end{defn}

	
	Denote by $[\underline{E}, \phi]\in \widehat{K}^0_{}(B)$
	the equivalence class of a cycle $(\underline{E}, \phi)$.
	For $[\underline{E}, \phi], [\underline{F}, \psi]\in 
	\widehat{K}^0_{}(B)$, 
	\begin{align}\label{eq:2.084}
		(\underline{E}, \phi)\cup(\underline{F}, \psi)
		:=(\underline{E}\otimes \underline{F}, 
		\ch(\underline{E})\wedge\psi+\phi\wedge\ch(\underline{F})
		-d\phi\wedge\psi  )
	\end{align} 
	induces a  
	well-defined,  commutative and
	associative product on $\widehat{K}^0_{}(B)$ (cf. \cite[Proposition 4.79]{K08}, also \cite[Lemma 2.15]{LM20}). 
	Thus $(\widehat{K}^0_{}(B), +, \cup)$ is a commutative ring
	with unity $1:=[\underline{\mathbb{C}}, 0]$, where
	$\underline{\mathbb{C}}$ denotes the trivial line bundle
	over $B$ with trivial metric and connection.
	
	From \cite[(2.21)]{FreedLott10}, there exists an exact 
	sequence
	\begin{align}\label{d034}
		K^{1}(B)\overset{\footnotesize{\ch}}{\longrightarrow} 
		\Omega^{\mathrm{odd}}(B, \R)/\Im\, d
		\overset{\footnotesize{a}}{\longrightarrow} \widehat{K}^{0}(B) 
		\overset{\footnotesize{I}}{\longrightarrow} K^{0}(B)
		\longrightarrow 0,
	\end{align}		
	where 
	\begin{align}
		\ch: K^{1}(B)\to H^{\mathrm{odd}}(B,\R)\subset 
		\Omega^{\mathrm{odd}}(B, \R)/\Im\, d
	\end{align}
	is the canonical Chern character map for $K^1$-group
	(cf. also \cite[\S 3.1]{LM20}) and
	\begin{align}\label{eq:2.10}
		a(\phi):=[\underline{0}, \phi], \quad I([\underline{E},\phi])
		:=[E]\in K^0(B).
	\end{align}
	
	For a complex vector bundle $E$ over $B$, let $\mathbb{P}(E)$
	be the projective bundle over $B$ associated with $E$,
	which is the quotient space obtained from $E\backslash 
	B$ by modulo the action of the complex number multiplication. Denote by $\sigma: \mathbb{P}(E)\to B$
	the natural projection. Then it induces a pull-back
	map 
	\begin{align}\label{eq:2.08a}
		\sigma^*:\widehat{K}^0(B)\to \widehat{K}^0(\mathbb{P}(E))
	\end{align} 
	in a natural way, which is a pre-$\lambda$-ring homomorphism by (\ref{eq:2.084}) (cf. \cite[Lemma 2.7]{LM20}).
	
	The splitting principle is a very useful tool in usual
	K-theory. The following result guarantees that the splitting
	principle can still be applied in differential K-theory.

	\begin{thm}\label{prop:2.2}
		The homomorphism $\sigma^*$ in (\ref{eq:2.08a}) is injective.
	\end{thm}
	\begin{proof}
		From (\ref{d034}), easy to see that
		the pull-back map $\sigma^*$
		induces a commutative diagram
		\begin{equation}\label{eq:2.12}
			\begin{split}
				\xymatrix{
					K_{}^{1}(B) \ar[d]^{\sigma^*} \ar[r]^-{\ch}
					&\Omega^{\mathrm{odd}}(B, \R)/\Im\, d
					\ar[d]^{\sigma^*} \ar[r]^-{a}
					&\widehat{K}^{0}(B)\ar[d]^{\sigma^*} 
					\ar[r]^-{I} &K^{0}(B)\ar[d]^{
						\sigma^*} \ar[r] &0\\
					K^{1}(\mathbb{P}(E))  \ar[r]^-{\ch}
					&\Omega^{\mathrm{odd}}(\mathbb{P}(E), \R)/\Im\, d
					\ar[r]^-{a} 
					&\widehat{K}^{0}(\mathbb{P}(E))\ar[r]^-{I} 
					&K_{}^{0}(\mathbb{P}(E))\ar[r] &0,
				}
			\end{split}
		\end{equation}
		in which the rows are exact sequences.
		By \cite[Corollary 2.7.9]{A67}, $\sigma^*:K^*(B)\to K^*(\mathbb{P}(E))$ is injective.
		
		For $\hat{x}\in \widehat{K}^0(B)$, if $\sigma^*(\hat{x})=0$,
		from (\ref{eq:2.12}), we have 
		\begin{align}
			\sigma^*\circ I(\hat{x})
			=I\circ \sigma^*(\hat{x})=0.
		\end{align}
		Since $\sigma^*|_{K^0(B)}$
		is injective, we have $I(\hat{x})=0$. Thus by (\ref{d034}),
		there exists $\phi \in \Omega^{\mathrm{odd}}(B, \R)/\Im\, d$
		such that $\hat{x}=a(\phi)$. Then by (\ref{eq:2.12}),
		\begin{align}
			a(\sigma^*(\phi))=\sigma^*(a(\phi))=\sigma^*(\hat{x})=0.
		\end{align}
		Thus there exists $y\in K^1(\mathbb{P}(E))$ such that
		\begin{align}\label{eq:2.15a}
			\ch(y)=\sigma^*\phi.
		\end{align} 
		
		For any $m\in \mathbb{P}(E)$, it represents a one-dimensional subspace
		in the fiber of $E$. The union of these lines form
		a complex line bundle over $\mathbb{P}(E)$. We denote by $H$ the dual
		bundle of this line bundle. From \cite[Corollary 2.7.9]{A67},
		$K^1(\mathbb{P}(E))$ is a free $K^1(B)$-module with basis
		$1, [H], [H]^2,\cdots, [H]^{r-1}$ with relation
		$\sum_{i=0}^r(-1)^i[H]^i[\sigma^*(\Lambda^i(E))]=0$,
		where $r$ is the rank of $E$.
		Thus there exists $F_i\in K^1(B)$ such that
		\begin{align}
			y=\sum_{i=0}^{r-1}\sigma^*(F_i)\cdot [H]^{i}.
		\end{align}
		This implies that
		\begin{align}\label{eq:2.16}
			\sigma^*\phi=\ch(y)=\sum_{i=0}^{r-1}\sigma^*\ch(F_i)
			\,e^{ic_1(H)}.
		\end{align}
In particular, we get $\sigma^*d\phi=d\sigma^*\phi=0\in 
\Omega^{\bullet}(\mathbb{P}(E),\R)$, thus $d\phi=0$, i.e.,
$\phi\in H^{\mathrm{odd}}(B,\R)$.
		
Note that by the Leray-Hirsch theorem, $H^{\bullet}(\mathbb{P}(E),\R)$ is a free $H^{\bullet}(B, \R)$-module with basis $1, c_1(H), \cdots, c_1(H)^{r-1}$ with the relation
\begin{align}\label{eq:2.17b}
c_1(H)^r+\sum_{j=0}^{r-1}c_1(H)^jc_{r-j}(E)=0,
\end{align}		
with $c_j(E)$ the $j$-th Chern class of $E$.

We claim that $H^{\bullet}(\mathbb{P}(E),\R)$ is also a free $H^{\bullet}(B, \R)$-module with basis $1, e^{c_1(H)}, \cdots, e^{(r-1)c_1(H)}$. In fact, by (\ref{eq:2.17b}), for any $j\in \N^*$, $e^{jc_1(H)}$ is a linear combination of $1, c_1(H), \cdots, c_1(H)^{r-1}$
with coefficients in $H^{\bullet}(B,\R)$, and there exist $b_{kj}\in H^{\bullet}(B,\R)$ with $\deg b_{kj}>0$ such that
$e^{0c_1(H)}=1$ and for $j\geq 1$,
\begin{align}\label{eq:2.19b}
e^{jc_1(H)}=\sum_{k=0}^{r-1}j^k\frac{c_1(H)^k}{k!}+
\sum_{k\geq r}j^k\frac{c_1(H)^k}{k!}
=\sum_{k=0}^{r-1}\left(j^k+b_{kj}\right)\frac{c_1(H)^k}{k!}.
\end{align}		
Let $\mathbf{A}=\mathbf{V}+\mathbf{B}$ be the matrix with $
\mathbf{V}=(j^k)_{k,j=0}^{r-1}$ and $\mathbf{B}=(b_{kj})_{k,j=0}^{r-1}$. Set
\begin{align}
{\bf e}=\left(1, e^{c_1(H)}, \cdots, e^{(r-1)c_1(H)} \right),
 \quad
{\bf f}=\left(
	1, \frac{c_1(H)}{1!}, 	\cdots,	\frac{c_1(H)^{r-1}}{(r-1)!}
\right).
\end{align}
Then from (\ref{eq:2.19b}), we have ${\bf e}={\bf f}\mathbf{A}$. 
Since $\mathbf{V}$ is the Vandermonde matrix which is invertible
and $\mathbf{B}$ is nilpotent, $\mathbf{A}$ is 
invertible in $\mathrm{M}_n(H^{\bullet}(B, \R))$ and ${\bf f}={\bf e}\mathbf{A}^{-1}$.
We obtain the claim. 
		
From (\ref{eq:2.16}) and the claim above, we know that
for any $i>0$, $\sigma^*\ch(F_i)=0$ and 
$\sigma^*\phi=\sigma^*\ch(F_0)$. This implies
$\phi=\ch(F_0)$.		
%
%
%
%
%
%
%
%
%
%
%
%
%
%
%
		Thus by (\ref{d034}),
		\begin{align}
			\hat{x}=a(\phi)=a\circ \ch(F_0)
			=0.
		\end{align}
		So $\sigma^*$ in (\ref{eq:2.08a}) is injective.
		The proof of Theorem \ref{prop:2.2} is complete.
	\end{proof}

	We denote by 
	$Z^{\mathrm{even}}(B,\R)$ the vector space of even degree closed 
	differential forms on $B$. Consider the vector space
	\begin{align}
		\Gamma(B):=Z^{\mathrm{even}}(B,\R)\oplus (\Omega^{\mathrm{odd}}(B,\R)/
		\mathrm{Im}\,d).
	\end{align}
	We give degree $l\geq 0$ to 
	$Z^{2l}(B,\R)\oplus \left(\Omega^{2l-1}(B,\R)/\Im \, d\right)$
	with $\Omega^{-1}(\cdot)=\{0\}$. 
	We define the multiplication on $\Gamma(B)$ by the formula
	\begin{align}\label{eq:2.008}
		(\omega_1,\phi_1)*(\omega_2, \phi_2)
		:=(\omega_1\wedge\omega_2, 
		\omega_1\wedge\phi_2+\phi_1\wedge\omega_2
		-d\phi_1\wedge\phi_2),
	\end{align}
	which is commutative and associative. 
	Then $\Gamma(B)$ is a ring with unity $(1,0)$ which is torsion free. 
	We define the Adams operations 
	$\Psi_{\Gamma}^k:\Gamma(B)\rightarrow \Gamma(B)$ 
	for $k\in \N$ by
	\begin{align}\label{eq:2.009}
		\Psi_{\Gamma}^k(\alpha,\beta)=(k^l\alpha,k^l\beta)
		\quad\text{for}\quad 
		(\alpha,\beta)\in Z^{2l}(B,\R)\oplus 
		(\Omega^{2l-1}(B,\R)/\Im \, d).
	\end{align}
	Then the relation
	\begin{align}\label{eq:2.004}
		\lambda_t(x)=\sum_{n\geq 
			0}\lambda_{\Gamma}^n(x)t^n:=
		\exp\left(\sum_{k=1}^{\infty}\frac{(-1)^{k-1}
			\Psi_{\Gamma}^k(x)t^k}{k}\right)
	\end{align}
	gives a pre-$\lambda$-ring structure on $\Gamma(B)$ 
	(cf. \cite[(2.4)]{LM20}). 
	From (\ref{eq:2.008})-(\ref{eq:2.004}), for $\omega\in 
	Z^{\mathrm{even}}(B,\R)$, $\phi\in 
	\Omega^{\mathrm{odd}}(B,\R)/\Im \, d$, we have
	\begin{align}\label{eq:2.8a}
		\left[ \lambda_{\Gamma}^k(\omega, 0)\right]_{\mathrm{odd}}=0,
		\quad \left[ \lambda_{\Gamma}^k(0, \phi)\right]_{\mathrm{even}}=0,
	\end{align}			
	where $[\cdot]_{\mathrm{odd}}$ (resp. $[\cdot]_{\mathrm{even}}$) is the 
	component of $\Gamma(B)$ in 
	$\Omega^{\mathrm{odd}}(B,\R)/\Im \, d$ (resp. $Z^{\mathrm{even}}(B,\R)$).	
	
	For $\alpha\in Z^{2l}(B,\R)$, the map $\Psi_Z^k(\alpha)
	:=k^l\alpha$ also defines Adams operations $\Psi_Z^k:Z^{\mathrm{
			even}}(B,\R)\to Z^{\mathrm{
			even}}(B,\R)$ for $k\in \N$. It induces 
	a pre-$\lambda$-ring structure $\lambda_Z^n$ on $Z^{\mathrm{
			even}}(B,\R)$ as in (\ref{eq:2.004}).
	From \cite[Lemma 2.3]{LM20}, the restriction map 
	$p:\Gamma(B)\to Z^{\mathrm{
			even}}(B,\R)$ is a homomorphism of pre-$\lambda$-rings.
	
	\begin{lemma}\label{lem:2.1}
		The pre-$\lambda$-ring structures on $\Gamma(B)$ and $Z^{\mathrm{
				even}}(B,\R)$ are $\lambda$-ring structures.
		Moreover the restriction map $p:\Gamma(B)\to Z^{\mathrm{
				even}}(B,\R)$ is a homomorphism of $\lambda$-rings.
	\end{lemma}
	\begin{proof}
		Since $\Gamma(B)$ (resp. $Z^{\mathrm{
				even}}(B,\R)$) is torsion free and the Adams operations
		$\Psi_{\Gamma}^k$ (resp. $\Psi_Z^k$) satisfy 
		(\ref{eq:1.21}),
		from Theorem \ref{thm:1.07}, the pre-$\lambda$-ring
		structure on $\Gamma(B)$ (resp. $Z^{\mathrm{
				even}}(B,\R)$) is a $\lambda$-ring structure.
		Moreover, since $p\circ \Psi_{\Gamma}^k=\Psi_Z^k\circ p$,
		by (\ref{eq:2.004}) and the corresponding definition
		for $\lambda_Z^n$, the restriction map $p:\Gamma(B)\to Z^{\mathrm{
				even}}(B,\R)$ is a homomorphism of $\lambda$-rings.
		The proof of Lemma \ref{lem:2.1} is complete.
	\end{proof}

	Let $(\underline{E},\phi)$ and $(\underline{F},\psi)$ be cycles of 
	$\widehat{K}^0_{}(B)$. Then by (\ref{eq:2.084}) and (\ref{eq:2.008}),
	\begin{align}\label{eq:2.040}
		(\underline{E}, \phi)\cup(\underline{F}, \psi)
		=\Big(\underline{E}\otimes \underline{F}, 
		\big[(\ch(\underline{E}),\phi)*(\ch(\underline{F}),\psi)
		\big]_{\mathrm{odd}}\Big).
	\end{align} 
	Note that from \cite[Lemma 2.3, (2.11), (2.33) and (2.47)]{LM20}, we have
	\begin{align}\label{eq:2.054}
		\left[\lambda_{\Gamma}^k(\ch(\underline{E}),\phi)
		\right]_{\mathrm{even}}=\ch(\Lambda^k(\underline{E})),
	\end{align}
	here $\Lambda^k(\underline{E})$ denotes the $k$-th exterior 
	algebra bundle of $E$ with corresponding 
	metric and connection induced  from $\underline{E}$.
	Set
	\begin{align}\label{eq:2.043} 
		\lambda^k(\underline{E},\phi):=\left(\Lambda^k(\underline{E}), 
		\left[\lambda_{\Gamma}^k(\ch(\underline{E}),\phi)\right]_{
			\mathrm{odd}}\right).
	\end{align}	
	By (\ref{eq:2.8a})  and (\ref{eq:2.043}),
	\begin{align}\label{eq:2.10a}
		\lambda^k(\underline{0}, \phi)=\left(\underline{0}, 
		\left[\lambda_{\Gamma}^k(0,\phi)\right]_{
			\mathrm{odd}}\right),\quad 
		\lambda^k(\underline{E},0):=\left(\Lambda^k(\underline{E}), 
		0\right).
	\end{align}
	From the proof of \cite[Theorem 2.6]{LM20}, the map
	$\lambda^k$ in (\ref{eq:2.043}) induces a well-defined
	map 
	\begin{align}
		\lambda^k:\widehat{K}_{}^0(B)\to \widehat{K}_{}^0(B),
	\end{align}
	that is, if $[\underline{E}, \phi]=[\underline{F}, \psi]
	\in \widehat{K}^0(B)$, we have
	\begin{align}\label{eq:2.28a}
		\left[ \lambda^k(\underline{E},\phi)\right]=
		\left[ \lambda^k(\underline{F},\psi)\right]
		\in \widehat{K}^0(B).
	\end{align}
	Moreover by \cite[Theorem 2.6]{LM20}, (\ref{eq:2.043}) induces a 
	pre-$\lambda$-ring structure on $\widehat{K}_{}^0(B)$.

	\
	
	The main result of this paper is as follows.

	\begin{thm}\label{thm:1.09}
		The pre-$\lambda$-ring structure on $\widehat{K}_{}^0(B)$  
		defined in (\ref{eq:2.043})
		is a $\lambda$-ring structure.
	\end{thm}
	
	
	\section{Proof of Theorem \ref{thm:1.09}}\label{s03}
	
			
			Recall that $1+\widehat{K}^0(B)[[t]]^+$ is a pre-$\lambda$-ring with operations $(\widetilde{+},
			\widetilde{\times}, \tilde{\lambda}^n)$ in (\ref{eq:2.041})
			and (\ref{eq:1.022}).
			
			\begin{lemma}\label{lemma:3.1}
				For cycles $(\underline{E},\phi)$ and $(\underline{F},\psi)$, the following identities hold
				for cycles:
				\begin{align}
					\lambda_t\big((\underline{E}, 0)\cup (\underline{0},\psi)\big)
					&= \lambda_t(\underline{E}, 0) \mathbin{\widetilde{\times}} \lambda_t(\underline{0},\psi), \label{eq:2.23a} \\
					\lambda_t\big((\underline{0},\phi) \cup (\underline{F}, 0)\big)
					&= \lambda_t(\underline{0},\phi) \mathbin{\widetilde{\times}} \lambda_t(\underline{F}, 0), \label{eq:2.23b} \\
					\lambda_t\big((\underline{0}, \phi)\cup (\underline{0},\psi))
					&= \lambda_t(\underline{0}, \phi) \mathbin{\widetilde{\times}} \lambda_t(\underline{0},\psi), \label{eq:2.24a} \\
					\lambda_t\left(\lambda^m(\underline{0}, \phi)\right)
					&= \tilde{\lambda}^m\left(\lambda_t(\underline{0}, \phi)\right). \label{eq:2.24b}
				\end{align}
			\end{lemma}

			\begin{proof}

				By (\ref{eq:2.084}), (\ref{eq:2.008}) and (\ref{eq:2.043}),
				\begin{multline}\label{eq:1.037}
					\lambda^n\big((\underline{E}, 0)\cup (\underline{0},\psi )\big)
					\stackrel{(\ref{eq:2.084})}{=}\lambda^n\big(\underline{0}, 
					\ch(\underline{E})\psi \big)
					\stackrel{(\ref{eq:2.043})}{=}\Big(\underline{0}, \big[
					\lambda_{\Gamma}^n(0, \ch(\underline{E})\psi)\big]_{\mathrm{odd}}\Big)
					\\
					\stackrel{(\ref{eq:2.008})}{=}\Big(\underline{0}, \big[
					\lambda_{\Gamma}^n((\ch(\underline{E}),0)*(
					0, \psi))\big]_{\mathrm{odd}}\Big).
				\end{multline}
				From (\ref{eq:2.8a}) and (\ref{eq:2.054}), 
				for $k\in \N$,
				we have
				\begin{align}\label{eq:1.49}
					\lambda^k_{\Gamma}(\ch(\underline{E}), 0)
					=(\ch(\Lambda^k(\underline{E})), 0),\quad
					\lambda^k_{\Gamma}(0, \psi)=	\big(0,
					[\lambda^k_{\Gamma}(0, \psi)]_{\mathrm{odd}}\big).
				\end{align}
				Recall that $P_n$ is a polynomial with weight $n$ in $\{\sigma_i\}$.
				Since $\Gamma(B)$ is a $\lambda$-ring, by (\ref{eq:1.23}), 
				(\ref{eq:2.008}) and (\ref{eq:1.49}),
				for $n\geq 1$, we get
				\begin{multline}\label{eq:2.20a}
					\lambda_{\Gamma}^n\big((\ch(\underline{E}),0)*(0, \psi)\big)
					\stackrel{(\ref{eq:1.23})}{=}P_n\big(
					\lambda_{\Gamma}^1(\ch(\underline{E}), 0), \cdots, \lambda_{\Gamma}^n(
					\ch(\underline{E}), 0); \lambda_{\Gamma}^1(0, \psi),\cdots, 
					\lambda_{\Gamma}^n(0, \psi)\big)
					\\
					\stackrel{(\ref{eq:1.49})}{=}
					P_n\big(
					(\ch(\Lambda^1(\underline{E})), 0), \cdots, (
					\ch(\Lambda^n(\underline{E})), 0); \left(0, 
					\left[\lambda_{\Gamma}^1(0, \psi) \right]_{\mathrm{odd}}\right),\cdots, 
					\left(0, 
					\left[\lambda_{\Gamma}^n(0, \psi) \right]_{\mathrm{odd}}\right)\big)
					\\
					\stackrel{(\ref{eq:2.008})}{=}\left( 0, \left[P_n\big(
					(\ch(\Lambda^1(\underline{E})), 0), \cdots, (
					\ch(\Lambda^n(\underline{E})), 0);\right.\right.
					\\
					\left.\left. \left(0, 
					\left[\lambda_{\Gamma}^1(0, \psi) \right]_{\mathrm{odd}}\right),\cdots, 
					\left(0, 
					\left[\lambda_{\Gamma}^n(0, \psi) \right]_{\mathrm{odd}}\right)\big) 
					\right]_{\mathrm{odd}}\right).
				\end{multline}
				By (\ref{eq:2.040}), (\ref{eq:2.10a})
				and (\ref{eq:2.20a}), for $n\geq 1$,
				\begin{align}\label{eq:2.21a}
					\begin{split}
						&P_n\big(\lambda^1(
						\underline{E}, 0), \cdots, \lambda^n(\underline{E}, 0); \lambda^1(
						\underline{0}, \psi), \cdots, \lambda^n(\underline{0}, \psi) \big)
						\\
						\stackrel{(\ref{eq:2.10a})}{=}&P_n\left((
						\Lambda^1(\underline{E}), 0), \cdots, 
						(\Lambda^n(\underline{E}), 0);  \left(\underline{0}, 
						\left[\lambda_{\Gamma}^1(0, \psi) \right]_{\mathrm{odd}}\right), 
						\cdots,  \left(\underline{0}, 
						\left[\lambda_{\Gamma}^n(0, \psi) 
						\right]_{\mathrm{odd}}\right) \right)
						\\
						\stackrel{(\ref{eq:2.040})}{=}&\left( \underline{0}, \left[P_n\big(
						(\ch(\Lambda^1(\underline{E})), 0), \cdots, (
						\ch(\Lambda^n(\underline{E})), 0); \right.\right.
						\\
						&\quad\quad\quad\left.\left.\left(0, 
						\left[\lambda_{\Gamma}^1(0, \psi) \right]_{\mathrm{odd}}\right),\cdots, 
						\left(0, 
						\left[\lambda_{\Gamma}^n(0, \psi) 
						\right]_{\mathrm{odd}}\right)\big) \right]_{\mathrm{odd}}\right)
						\\
						\stackrel{(\ref{eq:2.20a})}{=}&\left( \underline{0}, 
						\left[\lambda_{\Gamma}^n\big((\ch(\underline{E}),0)*(0, 
						\psi)\big) \right]_{\mathrm{odd}}\right).
					\end{split}
				\end{align}
				By (\ref{eq:1.037})  and (\ref{eq:2.21a}),  we get
				\begin{align}\label{eq:1.54}
					\lambda^n\big((\underline{E}, 0)\cup (\underline{0},\psi )\big)
					=P_n\big(\lambda^1(\underline{E}, 0), \cdots, \lambda^n(
					\underline{E}, 0); \lambda^1(\underline{0}, \psi),
					\cdots, \lambda^n(\underline{0}, \psi) \big).
				\end{align}
				Similarly, we have
				\begin{align}\label{eq:1.041}
					\lambda^n\big((\underline{0}, \phi)\cup (\underline{0},\psi ))
					=P_n\big(\lambda^1(\underline{0}, \phi), \cdots, \lambda^n(
					\underline{0}, \phi); \lambda^1(\underline{0}, \psi),
					\cdots, \lambda^n(\underline{0}, \psi) \big).
				\end{align}
				Thus from (\ref{eq:2.001}), (\ref{eq:2.041}), (\ref{eq:1.54}) 
				and (\ref{eq:1.041}), we obtain (\ref{eq:2.23a})
				and (\ref{eq:2.24a}).
				Since the operations ``$\cup$" and ``$\widetilde{\times}$"
				are commutative, from (\ref{eq:2.23a}),
				we get (\ref{eq:2.23b}).

				On the other hand,
				by (\ref{eq:2.040})-(\ref{eq:2.043}), we have
				\begin{align}\label{eq:2.79}
					\lambda^i\big(\underline{E}, \phi\big)\cup \lambda^j
					\big(\underline{F}, \psi\big)
					=\Big(\Lambda^i(\underline{E})\otimes \Lambda^j(\underline{F}), 
					\big[\lambda^i_{\Gamma}(\ch(\underline{E}), \phi)*
					\lambda^j_{\Gamma}(\ch(\underline{F}), \psi)\big]_{\mathrm{odd}}\Big).
				\end{align}
				By (\ref{eq:2.043}),
				\begin{align}\label{eq:2.80}
					\lambda^i\big(\underline{E}, \phi\big)+ \lambda^j
					\big(\underline{F}, \psi\big)
					=\Big(\Lambda^i(\underline{E})\oplus \Lambda^j(\underline{F}), 
					\big[\lambda^i_{\Gamma}(\ch(\underline{E}), \phi)+
					\lambda^j_{\Gamma}(\ch(\underline{F}), \psi)\big]_{\mathrm{odd}}\Big).
				\end{align}
				Since $P_{n,m}$ is a polynomial,
				from (\ref{eq:2.79}) and (\ref{eq:2.80}), we have
				\begin{multline}\label{eq:2.39}
					P_{n,m}\big(\lambda^1(\underline{E},\phi),\cdots,
					\lambda^{mn}(\underline{E},\phi)\big)
					\\
					=\Big(P_{n,m}\big(\Lambda^1(\underline{E}),\cdots,
					\Lambda^{mn}(\underline{E})\big), \Big[
					P_{n,m}\big(\lambda_{\Gamma}^1(\ch(\underline{E}),\phi),\cdots,
					\lambda_{\Gamma}^{mn}(\ch(\underline{E}),\phi)\big)
					\Big]_{\mathrm{odd}} \Big).
				\end{multline} 
				Recall that $\Gamma(B)$ is a 
				$\lambda$-ring. By (\ref{eq:1.34}) and (\ref{eq:2.39}), we see that
				\begin{multline}\label{eq:1.044}
					P_{n,m}\big(\lambda^1(\underline{E},\phi),\cdots,
					\lambda^{mn}(\underline{E},\phi)\big)
					\\
					=\Big(P_{n,m}\big(\Lambda^1(\underline{E}),\cdots,
					\Lambda^{mn}(\underline{E})\big),
					\big[\lambda_{\Gamma}^n(\lambda_{\Gamma}^m(\ch(\underline{E}),
					\phi))\big]_{\mathrm{odd}} \Big).
				\end{multline}
				By (\ref{eq:2.001}),
				(\ref{eq:1.022}),  (\ref{eq:2.054}), (\ref{eq:2.043}) and (\ref{eq:1.044}), 
				\begin{multline}\label{eq:1.62}
					\lambda_t\big(\lambda^m(\underline{0},\phi)\big)
					\stackrel{(\ref{eq:2.043})}{=}\lambda_t\Big(
					\underline{0}, \big[\lambda_{\Gamma}^m(0, \phi)\big]_{\mathrm{odd}}\Big)
					\\
					\xlongequal{(\ref{eq:2.001})(\ref{eq:2.043})}1+\sum_{n\geq 1}
					t^n \Big(\underline{0}
					, \Big[\lambda_{\Gamma}^n\big(0, 
					\big[\lambda_{\Gamma}^m(0, \phi)\big]_{\mathrm{odd}} \big) \Big]_{\mathrm{odd}} \Big)
					\stackrel{(\ref{eq:2.054})}{=}1+\sum_{n\geq 1}t^n \Big(\underline{0}
					, \Big[\lambda_{\Gamma}^n\big(\lambda_{\Gamma}^m\big(
					0,\phi\,\big)
					\big)\Big]_{\mathrm{odd}}\Big)
					\\
					\stackrel{(\ref{eq:1.044})}{=}1+\sum_{n\geq 1}t^n 
					P_{n,m}\big(\lambda^1(\underline{0},\phi),\cdots,
					\lambda^{mn}(\underline{0},\phi)\big)\stackrel{(\ref{eq:1.022})}{=}
					\tilde{\lambda}^m
					\big(\lambda_t(\underline{0},\phi)\big).
				\end{multline}
				We obtain (\ref{eq:2.24b}).	
				The proof of Lemma \ref{lemma:3.1} is complete.
			\end{proof}
			
			By applying the splitting principle Theorem \ref{prop:2.2}
			in differential K-theory, we establish the following 
			key part for Theorem \ref{thm:1.09}.
			
			\begin{lemma}\label{lemma:3.2}
				For cycles $(\underline{E},0)$ and $(\underline{F}, 0)$, we have
				\begin{align}\label{eq:2.056}
					\lambda_t\left([\underline{E},0]\cup [\underline{F},0]\right)
					=[\lambda_t(\underline{E},0)]\widetilde{\times}[\lambda_t
					(\underline{F},0)]\in \widehat{K}^0(B),
				\end{align}
				and
				\begin{align}\label{eq:2.057}
					\lambda_t([\lambda^m(\underline{E},0 )])=\widetilde{\lambda}^m
					\left[\lambda_t(\underline{E},0)\right]\in \widehat{K}^0(B).
				\end{align}
			\end{lemma}
			\begin{proof}
				We will prove this lemma by induction on the rank of the vector bundle $E$, which we denote by $\mathrm{rk} E$.
				
				If $\mathrm{rk} E=1$, by (\ref{eq:1.01}) and (\ref{eq:2.10a}), we have
				\begin{align}\label{eq:3.19}
					\lambda^0([\underline{E}, 0])=1,\quad 
					\lambda^1([\underline{E}, 0])=[\underline{E}, 0],\quad
					\lambda^m([\underline{E}, 0])=0\ \text{for}\ m>1.
				\end{align}
				From the definition of $P_n$, we see that
				\begin{align}\label{eq:3.20}
					P_n(s_1, 0\cdots, 0; \sigma_1,\cdots,\sigma_n)=s_1^n\sigma_n.
				\end{align}
				Thus from (\ref{eq:2.084}), (\ref{eq:2.10a}), (\ref{eq:3.19}) and (\ref{eq:3.20}), 
				\begin{multline}\label{eq:3.21}
					\lambda^m\left((\underline{E},0)\cup (\underline{F},0)\right)
					\stackrel{(\ref{eq:2.084})}{=}\lambda^m(\underline{E}\otimes \underline{F}, 0)
					\stackrel{(\ref{eq:2.10a})}{=}(\Lambda^m(\underline{E}\otimes \underline{F}), 0)
					=(\underline{E^m}\otimes \Lambda^m(\underline{F}), 0)
					\\
					\stackrel{(\ref{eq:2.084})}{=}
					(\underline{E},0)^m\cup (\Lambda^m(\underline{F}), 0)
					\xlongequal{(\ref{eq:3.19})(\ref{eq:3.20})}P_m\left((\underline{E},0), 0,\cdots, 0; (\Lambda^1(\underline{F}), 0), \cdots, (\Lambda^m(\underline{F}), 0) \right).
				\end{multline}
				Then (\ref{eq:2.056}) follows from (\ref{eq:2.001}),
				(\ref{eq:2.041}), (\ref{eq:3.19}) and (\ref{eq:3.21})
				for $\mathrm{rk} E=1$.
				
				We assume that (\ref{eq:2.056}) holds for any closed manifold
				$B$ with $\mathrm{rk} E\leq k$. 
				
				For a geometric triple $\underline{E}$ on $B$ with $\mathrm{rk} E=k+1$, we consider the projective bundle $
				\sigma: \mathbb{P}(E)\to B$
				of $E$. Let $\sigma^*\underline{E}$ be the pull-back
				of the geometric triple $\underline{E}$ on $\mathbb{P}(E)$.
				Then there exists a complex vector bundle $E_1$ on $\mathbb{P}(E)$
				such that 
				\begin{align}\label{eq:3.22}
					\sigma^*E=H^*\oplus E_1,
				\end{align}
				where $H^*$ is the tautological line bundle
				and the dual bundle of $H$ in the proof
				of Theorem \ref{prop:2.2}. Let $\underline{H^*}$ (resp.
				$\underline{E_1}$) be any geometric triple on $\mathbb{P}(E)$ with 
				underlying vector bundle $H^*$ (resp. $E_1$). From 
				Definition \ref{defn:2.12}, we have
				\begin{align}\label{eq:3.23}
					[\sigma^*\underline{E}, 0]=\left[\underline{H^*}\oplus \underline{E_1}, \widetilde{\ch}(\sigma^*\underline{E},
					\underline{H^*}\oplus \underline{E_1})\right]\in 
					\widehat{K}^0(\mathbb{P}(E)).
				\end{align}
				Now on $\mathbb{P}(E)$,
				\begin{multline}\label{eq:3.24}
					[\sigma^*\underline{E}, 0]\cup [\sigma^*\underline{F}, 0]
					=[\underline{H^*}, 0]\cup [\sigma^*\underline{F}, 0]
					+\left[\underline{E_1}, 0\right]\cup [\sigma^*\underline{F}, 0]
					\\	
					+\left[\underline{0},\widetilde{\ch}(\sigma^*\underline{E},
					\underline{H^*}\oplus \underline{E_1}) \right]\cup [\sigma^*\underline{F}, 0].
				\end{multline}
				From the inductive assumption, since $\mathrm{rk} H^*=1$ and
				$\mathrm{rk} E_1=k$, we have
				\begin{align}\label{eq:3.25a}
					\begin{split}
						\lambda_t\left([\underline{H^*}, 0]\cup [\sigma^*\underline{F}, 0]\right)&=\lambda_t([\underline{H^*}, 0])\widetilde{\times}
						\lambda_t([\sigma^*\underline{F}, 0]),
						\\
						\lambda_t\left(\left[\underline{E_1}, 0\right]\cup [\sigma^*\underline{F}, 0]\right)&=\lambda_t\left(\left[\underline{E_1}, 0\right]\right)\widetilde{\times}
						\lambda_t([\sigma^*\underline{F}, 0]).
					\end{split}
				\end{align}
				
				From (\ref{eq:2.002}), (\ref{eq:2.041}), (\ref{eq:1.05a}), (\ref{eq:2.08a}),
				(\ref{eq:2.23b}), (\ref{eq:3.23}), (\ref{eq:3.24})
				and (\ref{eq:3.25a}), we have
				\begin{multline}\label{eq:3.25}
					\sigma^*\lambda_t\left([\underline{E}, 0]\cup [\underline{F}, 0]\right)
					\stackrel{(\ref{eq:2.08a})}{=}\lambda_t\left([\sigma^*\underline{E}, 0]\cup [\sigma^*\underline{F}, 0]\right)
					\\
					\xlongequal{(\ref{eq:2.002})(\ref{eq:3.24})}\lambda_t\left([\underline{H^*}, 0]\cup [\sigma^*\underline{F}, 0]\right)\cdot \lambda_t\left(\left[\underline{E_1}, 0\right]\cup [\sigma^*\underline{F}, 0]\right)\cdot \lambda_t\left(
					\left[\underline{0},\widetilde{\ch}(\sigma^*\underline{E},
					\underline{H^*}\oplus \underline{E_1}) \right]\cup [\sigma^*\underline{F}, 0]\right)
					\\
					\xlongequal{(\ref{eq:2.041})(\ref{eq:2.23b})(\ref{eq:3.25a})}\lambda_t([\underline{H^*}, 0])\widetilde{\times}
					\lambda_t([\sigma^*\underline{F}, 0])\widetilde{+}
					\lambda_t\left(\left[\underline{E_1}, 0\right]\right)\widetilde{\times} \lambda_t([\sigma^*\underline{F}, 0])
					\\
					\widetilde{+}
					\lambda_t\left(\left[\underline{0},\widetilde{\ch}(\sigma^*\underline{E},
					\underline{H^*}\oplus \underline{E_1}) \right]\right)\widetilde{\times} \lambda_t([\sigma^*\underline{F}, 0])
					\\
					\stackrel{(\ref{eq:1.05a})}{=}\left(\lambda_t([\underline{H^*}, 0])\widetilde{+}
					\lambda_t\left(\left[\underline{E_1}, 0\right]\right)\widetilde{+} 
					\lambda_t\left(\left[\underline{0},\widetilde{\ch}(\sigma^*\underline{E},
					\underline{H^*}\oplus \underline{E_1}) \right]\right)\right)
					\widetilde{\times} \lambda_t([\sigma^*\underline{F}, 0])
					\\
					\xlongequal{(\ref{eq:2.002})(\ref{eq:2.041})}\lambda_t
					\left(\left[\underline{H^*}\oplus \underline{E_1}, \widetilde{\ch}(\sigma^*\underline{E},
					\underline{H^*}\oplus \underline{E_1})\right]\right)
					\widetilde{\times} \lambda_t([\sigma^*\underline{F}, 0])
					\\
					\stackrel{(\ref{eq:3.23})}{=}\lambda_t([\sigma^*\underline{E}, 0])
					\widetilde{\times} \lambda_t([\sigma^*\underline{F}, 0])
					=\sigma^*\left(\lambda_t([\underline{E}, 0])
					\widetilde{\times} \lambda_t([\underline{F}, 0])  \right)
					\in \widehat{K}^0(\mathbb{P}(E)).
				\end{multline}
				From Theorem \ref{prop:2.2}, $\sigma^*:\widehat{K}^0(B)
				\to \widehat{K}^0(\mathbb{P}(E))$ is injective. So (\ref{eq:2.056})
				holds for $\mathrm{rk} E=k+1$. The inductive process for 
				(\ref{eq:2.056}) is complete.
				
				The proof of (\ref{eq:2.057}) is similar.
				If $\mathrm{rk} E=1$, by (\ref{eq:2.001}), (\ref{eq:1.06a}) and (\ref{eq:3.19}), we have
				\begin{align}\label{eq:3.26}
					\lambda_t([\lambda^0(\underline{E}, 0)])\stackrel{(\ref{eq:3.19})}{=}\lambda_t(1)
					\stackrel{(\ref{eq:3.19})}{=}1+t \stackrel{(\ref{eq:1.06a})}{=}\tilde{\lambda}^0(1+t [\underline{E}, 0])
					\stackrel{(\ref{eq:3.19})}{=}\tilde{\lambda}^0(\lambda_t([\underline{E}, 0]))
				\end{align}
				and
				\begin{multline}\label{eq:3.27}
					\lambda_t([\lambda^1(\underline{E}, 0)])
					\stackrel{(\ref{eq:3.19})}{=}\lambda_t([\underline{E}, 0])
					\xlongequal{(\ref{eq:2.001})(\ref{eq:3.19})}1+t [\underline{E}, 0]
					\stackrel{(\ref{eq:1.06a})}{=}\tilde{\lambda}^1(1+t [\underline{E}, 0])
					\\
					\xlongequal{(\ref{eq:2.001})(\ref{eq:3.19})}\tilde{\lambda}^1(\lambda_t([\underline{E}, 0])).
				\end{multline}
				From the definition of $P_{n,m}$, we see that for $m>1$, $n\geq 1$,
				\begin{align}\label{eq:3.27a}
					P_{n,m}(s_1, 0,\cdots, 0)=0.
				\end{align}
				For $m>1$, from (\ref{eq:1.05b}), (\ref{eq:1.022}), (\ref{eq:3.19})
				and (\ref{eq:3.27a}),
				\begin{multline}\label{eq:3.28a}
					\tilde{\lambda}^m\left(\lambda_t([\underline{E}, 0])\right)
					\stackrel{(\ref{eq:3.19})}{=} \tilde{\lambda}^m\left(1+t[\underline{E}, 0]\right)
					\stackrel{(\ref{eq:1.022})}{=}
					1+\sum_{n\geq 1}P_{n,m}([\underline{E}, 0], 0,\cdots, 0) t^n
					\stackrel{(\ref{eq:3.27a})}{=} 1
					\\
					\stackrel{(\ref{eq:1.05b})}{=} \lambda_t(0)
					\stackrel{(\ref{eq:3.19})}{=} \lambda_t\left( 
					\left[\lambda^m([\underline{E}, 0])\right]\right).
				\end{multline}
				Thus (\ref{eq:2.057}) holds for $\mathrm{rk} E=1$.
				
				We assume that (\ref{eq:2.057}) holds for any closed manifold
				$B$ with $\mathrm{rk} E\leq k$. 
				For a geometric triple $\underline{E}$ on $B$ with $\mathrm{rk} E=k+1$, from (\ref{eq:1.01}) and (\ref{eq:3.23}),
				\begin{multline}\label{eq:3.28}
					\lambda^m([\sigma^*\underline{E}, 0])
					=\lambda^m\left(\left[\underline{H^*}\oplus \underline{E_1}, \widetilde{\ch}(\sigma^*\underline{E},
					\underline{H^*}\oplus \underline{E_1})\right] \right)
					\\
					=\sum_{i=0}^m \lambda^i\left(\left[\underline{H^*}\oplus \underline{E_1},0\right]\right)\cup \lambda^{m-i}\left(\left[\underline{0}, \widetilde{\ch}(\sigma^*\underline{E},
					\underline{H^*}\oplus \underline{E_1})  \right]\right).
				\end{multline}
				From (\ref{eq:2.002}), (\ref{eq:2.041}), 
				(\ref{eq:2.08a}),
				(\ref{eq:2.10a}), (\ref{eq:2.23a}), (\ref{eq:2.24b}) and (\ref{eq:3.28}),
				we get
				\begin{multline}\label{eq:3.29}
				\sigma^*\lambda_t(\lambda^m([\underline{E}, 0]))
				\stackrel{(\ref{eq:2.08a})}{=}	
					\lambda_t(\lambda^m([\sigma^*\underline{E}, 0]))
					\\
					\xlongequal{(\ref{eq:2.002})(\ref{eq:2.041})
						(\ref{eq:2.10a})(\ref{eq:2.23a})(\ref{eq:3.28})}
					\widetilde{+}_{i=0}^m\left(\lambda_t\left(\lambda^i\left(\left[\underline{H^*}\oplus \underline{E_1},0\right]\right)\right)\widetilde{\times}
					\lambda_t\left(\lambda^{m-i}\left(\left[\underline{0}, \widetilde{\ch}(\sigma^*\underline{E},
					\underline{H^*}\oplus \underline{E_1})  \right]\right)\right)
					\right)
					\\
					\stackrel{(\ref{eq:2.24b})}{=}\widetilde{+}_{i=0}^m
					\left(\lambda_t\left(\lambda^i\left(\left[\underline{H^*}\oplus \underline{E_1},0\right]\right)\right)\widetilde{\times}
					\tilde{\lambda}^{m-i}\left(\lambda_t\left(\left[\underline{0}, \widetilde{\ch}(\sigma^*\underline{E},
					\underline{H^*}\oplus \underline{E_1})  \right]\right)\right)\right).
				\end{multline}
				
				For $i=0$, from (\ref{eq:1.01}), (\ref{eq:1.06a}) and
				(\ref{eq:3.19}), we have
				\begin{align}
					\lambda_t\left( \lambda^0\left( \left[ \underline{H^*}\oplus \underline{E_1}, 0\right]\right)\right)
					\stackrel{(\ref{eq:1.01})}{=} \lambda_t(1)\stackrel{(\ref{eq:3.19})}{=} 1+t
					\stackrel{(\ref{eq:1.06a})}{=}
					\tilde{\lambda}^0\left( \lambda_t\left( \left[ \underline{H^*}\oplus \underline{E_1}, 0\right]\right)\right).
				\end{align}
				Since $\mathrm{rk} E_1=k$, from the inductive assumption,
				we have
				\begin{align}\label{eq:3.30a}
					\lambda_t\left( \lambda^i\left( \left[  \underline{E_1}, 0\right]\right)\right)=\tilde{\lambda}^i\left( \lambda_t\left( \left[ \underline{E_1}, 0\right]\right)\right).
				\end{align}
				Since $H^*$ is a line bundle, from (\ref{eq:1.01}), 
				(\ref{eq:2.002}), (\ref{eq:2.041}), (\ref{eq:1.06a}), 
				(\ref{eq:2.10a}),
				(\ref{eq:2.056}),
				(\ref{eq:3.19}), (\ref{eq:3.28a}) and (\ref{eq:3.30a}), for $i\geq 1$,
				\begin{multline}\label{eq:3.30}
					\lambda_t\left(\lambda^i\left(\left[\underline{H^*}\oplus \underline{E_1},0\right]\right)\right)
					\xlongequal{(\ref{eq:1.01})(\ref{eq:3.19})}
					\lambda_t\left(\lambda^i\left(\left[\underline{E_1}, 0\right]\right)+[\underline{H^*}, 0]\cup \lambda^{i-1}\left(\left[\underline{E_1}, 0\right]\right)\right)
					\\
					\stackrel{(\ref{eq:2.002})}{=}
					\lambda_t\left(\lambda^i\left(\left[\underline{E_1}, 0\right]\right)\right)\cdot
					\lambda_t\left([\underline{H^*}, 0]\cup \lambda^{i-1}\left(\left[\underline{E_1}, 0\right]\right)
					\right)
					\\
					\xlongequal{(\ref{eq:2.041})(\ref{eq:2.10a})(\ref{eq:2.056})}\lambda_t\left(\lambda^i\left(\left[\underline{E_1}, 0\right]\right)\right)
					\widetilde{+}\lambda_t([\underline{H^*}, 0])\widetilde{\times} \lambda_t\left(\lambda^{i-1}\left(\left[\underline{E_1}, 0
					\right]\right)\right)
					\\
					\xlongequal{(\ref{eq:1.06a})(\ref{eq:3.30a})}
					\tilde{\lambda}^i\left(\lambda_t\left(\left[\underline{E_1}, 0\right]\right)\right)
					\widetilde{+}\tilde{\lambda}^1\left(\lambda_t([\underline{H^*}, 0])\right)\widetilde{\times}\tilde{\lambda}^{i-1}\left(\lambda_t\left(
					\left[\underline{E_1}, 0\right]\right)\right)
					\\
					\xlongequal{(\ref{eq:1.06a})(\ref{eq:3.28a})}\tilde{\lambda}^i\left(\lambda_t([\underline{H^*}, 0])
					\widetilde{+}\lambda_t\left(\left[\underline{E_1}, 0\right]\right)\right)
					\xlongequal{(\ref{eq:2.002})(\ref{eq:2.041})}\tilde{\lambda}^i\left(\lambda_t\left(\left[\underline{H^*}\oplus \underline{E_1},0\right]\right)\right).
				\end{multline}
				So from (\ref{eq:2.002}), (\ref{eq:2.041}), (\ref{eq:1.06a}), (\ref{eq:2.08a}),
				(\ref{eq:3.23}), (\ref{eq:3.29}) and (\ref{eq:3.30}), we get
				\begin{multline}\label{eq:3.31}
					\sigma^*\left(\lambda_t\left(\lambda^m([\underline{E}, 0])\right)\right)
					\\
					\xlongequal{(\ref{eq:3.29})(\ref{eq:3.30})}
					\widetilde{+}_{i=0}^m\left(\tilde{\lambda}^i\left(\lambda_t
					\left(\left[\underline{H^*}\oplus \underline{E_1},0\right]\right)\right)\widetilde{\times}
					\tilde{\lambda}^{m-i}\left(\lambda_t\left(\left[\underline{0}, \widetilde{\ch}(\sigma^*\underline{E},
					\underline{H^*}\oplus \underline{E_1})  \right]\right)\right)
					\right)
					\\
					\stackrel{(\ref{eq:1.06a})}{=}\tilde{\lambda}^m\left(\lambda_t\left(\left[\underline{H^*}\oplus \underline{E_1},0\right]\right)\widetilde{+}
					\lambda_t\left(\left[\underline{0}, \widetilde{\ch}(\sigma^*\underline{E},
					\underline{H^*}\oplus \underline{E_1})  \right]\right)\right)
					\\
					\xlongequal{(\ref{eq:2.002})(\ref{eq:2.041})}\tilde{\lambda}^m\left(\lambda_t\left(\left[\underline{H^*}\oplus \underline{E_1}, \widetilde{\ch}(\sigma^*\underline{E},
					\underline{H^*}\oplus \underline{E_1})\right] \right)\right)
					\\
					\stackrel{(\ref{eq:3.23})}{=}\tilde{\lambda}^m(\lambda_t([\sigma^*\underline{E}, 0] ))
					=\tilde{\lambda}^m(\sigma^*\lambda_t([\underline{E}, 0] ))
					\stackrel{(\ref{eq:2.08a})}{=}\sigma^*\tilde{\lambda}^m(\lambda_t([\underline{E}, 0] )).
				\end{multline}
				Since $\sigma^*$ is injective, (\ref{eq:2.057}) holds for 
				$\mathrm{rk} E=k+1$. 
				The inductive process for (\ref{eq:2.057}) is complete.
			\end{proof}

			\begin{proof}[Proof of Theorem \ref{thm:1.09}]
				By Lemmas \ref{lemma:3.1} and \ref{lemma:3.2}, (\ref{eq:1.01}), (\ref{eq:2.002}) and (\ref{eq:2.041}), for 
				cycles $(\underline{E}, \phi)$ and $(\underline{F}, \psi)$,
				we have
				\begin{multline}\label{eq:2.057a}
					\lambda_t\big([\underline{E},\phi]\cup [\underline{F},\psi]\big)
					\xlongequal{(\ref{eq:2.002})(\ref{eq:2.041})}\lambda_t\big([\underline{E}, 0]\cup [\underline{F}, 0]\big)
					\widetilde{+}\lambda_t\big([\underline{E},0]\cup [\underline{0},\psi]
					\big)
					\\
					\widetilde{+}\lambda_t\big([\underline{0},\phi]\cup 
					[\underline{F},0]\big)
					\widetilde{+}\lambda_t\big([\underline{0},\phi]\cup 
					[\underline{0},\psi]\big)
					\\
					=\lambda_t\big([\underline{E}, 0]\big)\widetilde{\times} \lambda_t\big([\underline{F}, 0]\big)
					\widetilde{+}\lambda_t\big([\underline{E},0]\big)\widetilde{\times} \lambda_t\big( [\underline{0},\psi]
					\big)
					\\
					\widetilde{+}\lambda_t\big([\underline{0},\phi]\big)\widetilde{\times} \lambda_t\big( 
					[\underline{F},0]\big)
					\widetilde{+}\lambda_t\big([\underline{0},\phi]\big)\widetilde{\times} \lambda_t\big( 
					[\underline{0},\psi]\big)
					\\
					=\left(\lambda_t\big([\underline{E}, 0]\big)\widetilde{+}\lambda_t\big([\underline{0},\phi]\big)  \right)\widetilde{\times}\left(\lambda_t\big( 
					[\underline{F},0]\big)\widetilde{+} \lambda_t\big( 
					[\underline{0},\psi]\big) \right)
					=\lambda_t\big([\underline{E}, \phi]\big)
					\widetilde{\times}\lambda_t\big([\underline{F}, \psi]\big),
				\end{multline}
				and
				\begin{multline}\label{eq:3.33}
					\lambda_t(\lambda^m([\underline{E}, \phi]) )
					\stackrel{(\ref{eq:1.01})}{=}\lambda_t\left(\sum_{i=0}^m\lambda^i([\underline{E}, 0])\cup
					\lambda^{m-i}([\underline{0}, \phi]) \right)
					\\
					\xlongequal{(\ref{eq:2.002})(\ref{eq:2.041})}\widetilde{+}_{i=0}^m\left(\lambda_t\left(\lambda^i([\underline{E}, 0]) \right)
					\widetilde{\times} \lambda_t\left(\lambda^{m-i}([\underline{0}, \phi]) \right)\right)=\widetilde{+}_{i=0}^m\left(\tilde{\lambda}^i\left(\lambda_t([\underline{E}, 0]) \right)
					\widetilde{\times} \tilde{\lambda}^{m-i}\left(\lambda_t([\underline{0}, \phi]) \right)\right)
					\\
					=\tilde{\lambda}^m\left(\lambda_t([\underline{E}, 0])\widetilde{+}\lambda_t([\underline{0}, \phi])\right)
					=\tilde{\lambda}^m\left(\lambda_t([\underline{E}, \phi])\right).
				\end{multline}
				From Definition \ref{defn:1.06}, the pre-$\lambda$-ring
				structure on $\widehat{K}^0(B)$ is a $\lambda$-ring structure.
				The proof of Theorem \ref{thm:1.09} is complete.
			\end{proof}

			\section{Adams operations}\label{s04}
			
			If $R$ is a $\lambda$-ring, by \cite[\S V Proposition 7.5 i)]{BerthelotTh},
			for $k\in \N^*$, the Adams
			operations $\Psi^k: R\to R$
			defined in (\ref{eq:1.19}) satisfy (\ref{eq:1.21}).
			From \cite[\S I.5 (2)]{ATall69}, for $x\in R$,
			\begin{align}\label{eq:3.1a}
				\Psi^k(x)=\nu_k(\lambda^1(x),\cdots, \lambda^k(x)),
			\end{align}
			where $\nu_k(s_1,\cdots, s_k)=\xi_1^k+\cdots+\xi_q^k$
			is the $k$-th Newton polynomial, $s_i$ being the $i$-th
			elementary symmetric function in $\{\xi_j\}$.
			By \cite[\S I Proposition 5.1]{ATall69}, $\Psi^k: R\to R$
			is a homomorphism of $\lambda$-rings.
			
			For a geometric triple $\underline{E}$ on $B$, we denote by
			\begin{align}\label{eq:2.02}
				\Psi^k(\underline{E}):=\nu_k(\Lambda^1(\underline{E}),\cdots,
				\Lambda^k(\underline{E})).
			\end{align} 
			
			\begin{prop}\label{prop:3.1}
				For $[\underline{E}, \phi]\in \widehat{K}^0(B)$, 
				$\phi =\sum_{l\geq 1} \phi_l$, $\phi_l\in 
				\Omega^{2l-1}(B,\R)/\mathrm{Im}\,d$,
				$k\in \N^*$, we have
				\begin{align}\label{eq:2.03}
					\Psi^k([\underline{E}, \phi])
					=\left[ \Psi^k(\underline{E}), \sum_{l\geq 1} k^l\phi_l \right].
				\end{align}
				For any $k\in \N^*$, $\Psi^k:\widehat{K}^0(B)\to 
				\widehat{K}^0(B)$ is a homomorphism of $\lambda$-rings.	
			\end{prop}
			\begin{proof}
				Applying (\ref{eq:3.1a}) for the $\lambda$-ring
				$\Gamma(B)$, we get
				\begin{align}\label{eq:3.3a}
					\Psi_{\Gamma}^k(\ch(\underline{E}), \phi)=\nu_k(\lambda_{\Gamma}^1(
					\ch(\underline{E}), \phi),\cdots, \lambda_{\Gamma}^k(\ch(\underline{E}), \phi)).
				\end{align}
				Since $\widehat{K}^0(B)$ is a $\lambda$-ring
				and $\nu_k$ is a polynomial with integer
				coefficients,
				from (\ref{eq:2.79}), (\ref{eq:2.80}), (\ref{eq:3.1a}), 
				(\ref{eq:2.02}) and (\ref{eq:3.3a}),				
				we have
				\begin{multline}
					\Psi^k([\underline{E}, \phi])\stackrel{(\ref{eq:3.1a})}{=}
					\nu_k(\lambda^1([\underline{E}, \phi]),\cdots, \lambda^k([\underline{E}, \phi]))
					\\
					\xlongequal{(\ref{eq:2.79})(\ref{eq:2.80})}\left[
					\nu_k(\Lambda^1(\underline{E}), 
					\cdots, \Lambda^k(\underline{E})), 
					\left[\nu_k\left(\lambda_{\Gamma}^1(\ch(\underline{E}),\phi),
					\cdots,\lambda_{\Gamma}^k(\ch(\underline{E}),\phi) 
					\right)\right]_{\mathrm{odd}}\right]
					\\
					\xlongequal{(\ref{eq:2.02})(\ref{eq:3.3a})}
					\left[\Psi^k(\underline{E}), [\Psi_{\Gamma}^k(
					\ch(\underline{E}),\phi) ]_{\mathrm{odd}}
					\right].
				\end{multline}
				Thus this proposition follows directly from 
				(\ref{eq:2.009}). The last part is from 
				\cite[\S I Proposition 5.1]{ATall69} as explained above.
			\end{proof}
			
			Let $f: B\to B'$ be a smooth map between two closed manifolds. Then from Proposition \ref{prop:3.1}, for any $x\in \widehat{K}^0(B')$, we have
			\begin{align}\label{eq:4.06}
				\Psi^k(f^* x)=f^*\Psi^k(x).
			\end{align}			
			
			From (\ref{eq:2.02}), if $\underline{L}=(L, h^L, \nabla^L)$ 
			is a geometric triple on $B$ with line bundle $L$, 
			by the definition of $\nu_k$, we have
			\begin{align}\label{eq:3.6a}
				\Psi^k(\underline{L})=\underline{L^k}.
			\end{align}
			If $\underline{E}=\oplus_{j=1}^r\underline{L_j}$, 
			and $\underline{L_j}=(L_j, h^{L_j}, \nabla^{L_j})$
			with line bundles $L_j$, then by Theorem \ref{thm:1.09}, 
			Proposition \ref{prop:3.1} and  (\ref{eq:3.6a}),
			\begin{multline}\label{eq:3.2a}
				[\Psi^k(\underline{E}), 0]=\Psi^k([\underline{E},0])= 
				\Psi^k\left(\sum_{j=1}^r\left[\underline{L_j},0\right]\right)
				=\sum_{j=1}^r\Psi^k\left(\left[\underline{L_j},0
				\right]\right)
				\\
				=\sum_{j=1}^r\left[\Psi^k\left(\underline{L_j}\right),0\right]
				=\sum_{j=1}^r\left[\underline{L_j^k},0\right]
				=\left[\bigoplus_{j=1}^r\underline{L_j^k}, 0\right].
			\end{multline}
			%
			
			From Lemma \ref{lem:2.1} and (\ref{eq:2.054}),
			\begin{align}\label{eq:3.6b}
				\lambda_Z^k\left(\ch(\underline{E})\right)=\ch
				\left(\Lambda^k(\underline{E})\right).
			\end{align}
			Thus from (\ref{eq:3.1a}), (\ref{eq:2.02}), (\ref{eq:3.6b}) 
			and $\nu_k$ is a polynomial, we have
			\begin{multline}\label{eq:4.09}
				\Psi_{Z}^k\left(\ch(\underline{E})\right)
				=\nu_k\left(\lambda_Z^1\left(\ch(\underline{E})\right),\cdots,
				\lambda_Z^k\left(\ch(\underline{E})\right)\right)
				=\nu_k\left(\ch\left(\Lambda^1(\underline{E})\right),\cdots,
				\ch\left(\Lambda^k(\underline{E})\right)\right)
				\\
				=\ch\left(\nu_k\left(\Lambda^1(\underline{E}),\cdots,
				\Lambda^k(\underline{E})\right)\right)
				=\ch\left(\Psi^k(\underline{E})\right).
			\end{multline}
			Let $R: \widehat{K}^0(B)\to Z^{\mathrm{even}}(B, \mathbb{R})$ be the curvature map defined by
			\begin{align}\label{eq:4.10}
				R([\underline{E}, \phi]):=\ch(\underline{E})-d\phi.
			\end{align}			
			Then it is well-defined by Definition \ref{defn:2.12} and (\ref{eq:2.04}). Thus from Proposition \ref{prop:3.1}, (\ref{eq:4.09}) and
			(\ref{eq:4.10}), we have
			\begin{multline}\label{eq:4.12}
				R\circ \Psi^k([\underline{E}, \phi])=R\left(\left[ \Psi^k(\underline{E}), \sum_{l\geq 1} k^l\phi_l \right]\right)
				=\ch\left(\Psi^k(\underline{E}) \right)-\sum_{l\geq 1} k^l d\phi_l 
				\\
				=\Psi^k_Z(\ch(\underline{E}))-\Psi_Z^k(d\phi)=
				\Psi^k_Z\circ R([\underline{E}, \phi]).
			\end{multline}
			Moreover from Proposition \ref{prop:3.1} and (\ref{eq:2.10}),
			\begin{align}\label{eq:4.13}
				I\circ \Psi^k([\underline{E}, \phi])=\Psi^k\circ I([\underline{E}, \phi]).
			\end{align} 
			From the proof of \cite[Theorem 3.1]{Bunke10}, there is 
			a unique natural transform of $\Psi^k: \widehat{K}^0\to \widehat{K}^0$ which satisfies (\ref{eq:4.12}) and (\ref{eq:4.13}).
			So the Adams operations
			$\Psi^k$ in (\ref{eq:2.03}) gives a constructive realization
			of $\widehat{\Psi}^k$ in \cite{Bunke10}.

			\section{$T_g$-equivariant differential K-theory}\label{s05}
			
			Let $B$ be a closed smooth manifold and $G$ be a compact Lie group acting smoothly on $B$. In this subsection, we will discuss the $\lambda$-ring
			structure on some equivariant version of differential 
			K-theory on $B$ related to a fixed element of $G$.
			
			For $g\in G$ fixed, let $T_g:=\overline{\{g\}}$ be the 
			compact abelian subgroup of $G$ generated by $g$. Comparing 
			with the non-equivariant case, a $G$-equivariant geometric
			triple $\underline{E}=(E, h^E, \nabla^E)$ consists of a 
			$G$-equivariant Hermitian vector bundle $(E, h^E)$ and a 
			$G$-invariant Hermitian connection $\nabla^E$. We parametrize
			the complex irreducible representation of $T_g$ by $\{V_{\gamma} \}_{\gamma\in I}$, where each $V_{\gamma}$ is $1$-dimensional
			complex representation. Thus on 
			\begin{align}
				B^g:=\{x\in B; gx=x \},
			\end{align} 
			the fixed point set of $T_g$-action, as a $T_g$-equivariant
			geometric triple on $B^g$, we have the decomposition
			\begin{align}\label{eq:4.3}
				(E, h^E, \nabla^E)|_{B^g}=\bigoplus_{\gamma\in I} 
				(E_{\gamma}, h^{E_{\gamma}}, \nabla^{E_{\gamma}})\otimes V_{\gamma}.
			\end{align}
			
			Now we define
			\begin{align}\label{eq:4.4}
				\ch_{T_g}(\underline{E}):=\sum_{\gamma\in I}\ch\left(E_{\gamma}\right)\otimes 
				\chi_{V_{\gamma}}\in \Omega^{\mathrm{even}}(B^g, \R)\otimes 
				R(T_g),
			\end{align}
			with $\chi_{V_{\gamma}}$ the character of the representation $V_{\gamma}$.
			Here $R(T_g)$ denotes the representation ring of $T_g$.
			If we evaluate $\ch_{T_g}(\underline{E})$ at $g\in T_g$, we obtain the usual $\ch_g(
			\underline{E})$.
			
			Comparing with (\ref{eq:2.2}), for two $G$-equivariant 
			geometric triples $\underline{E_0}$ and $\underline{E_1}$,
			which have the same underlying $G$-equivariant vector bundle
			$E$, there exist the $G$-equivariant projection 
			$\pi: B\times \R\to B$, and $G$-equivariant geometric triple
			$\underline{\pi^*E}$ on $B\times \R$, such that 
			$\underline{\pi^*E}|_{B\times \{j\}}=\underline{E_j}$, $j=0, 1$.
			Here the $G$-action on $B\times \R$ is naturally lifted from 
			the $G$-action on $B$ and the trivial action on $\R$.
			Then on $B^g\times \R$, we have the decomposition as in 
			(\ref{eq:4.3}) for $\pi^*E$. Thus the equivariant Chern-Simons
			class $\widetilde{\ch}_{T_g}\left(\underline{E_0},
			\underline{E_1}\right)\in \left(\Omega^{\mathrm{odd}}(B^g, \R)
			/\mathrm{Im}\,d\right) \otimes R(T_g)$ is defined by
			\begin{align}
				\widetilde{\ch}_{T_g}\left(\underline{E_0},
				\underline{E_1}\right):=\int_0^1\left\{ 
				\ch_{T_g}\left(\underline{\pi^*E}\right)\right\}^{ds}ds.
			\end{align}
			Moreover, as in (\ref{eq:2.04}),
			\begin{align}\label{eq:5.5}
				d\,\widetilde{\ch}_{T_g}\left(\underline{E_0}, \underline{E_1}\right)
				=\ch_{T_g}(\underline{E_1})-\ch_{T_g}(\underline{E_0})
				\in \Omega^{\mathrm{even}}(B^g, \R)
				\otimes R(T_g)
			\end{align}
			and the Chern-Simons class depend only on $\nabla^{E_j}$
			for $j=0, 1$, which follows exactly from the proof of 
			\cite[Theorem B.5.4]{MM07}.
			
			\begin{defn}
				For $g\in G$, a cycle for the $T_g$-equivariant differential 
				K-theory of $B$ is a pair $(\underline{E}, \phi)$, where 
				$\underline{E}$ is a $G$-equivariant geometric triple on $B$
				and $\phi \in  \left(\Omega^{\mathrm{odd}}(B^g, \R)
				/\mathrm{Im}\,d\right) \otimes R(T_g)$. Two cycles
				$\left(\underline{E_1}, \phi_1\right)$ and $\left(\underline{E_2}, \phi_2\right)$ are equivalent
				if there exist a $G$-equivariant vector bundle isomorphism
				$\Phi:E_1\oplus E_3\to E_2\oplus E_3$ over $B$ such that
				\begin{align}
					\widetilde{\ch}_{T_g}\left(\underline{E_1}\oplus\underline{E_3}, 
					\Phi^{*} \left(\underline{E_2}\oplus
					\underline{E_3}\right)\right)
					=\phi_2-\phi_1.
				\end{align}
				We define the $T_g$-equivariant differential K-group 
				$\widehat{K}_{T_g}^0(B)$ to be the Grothendieck group of 
				equivalence classes of cycles. 
			\end{defn}
			
			For cycles $(\underline{E}, \phi)$ and $(\underline{F}, \psi)$,
			we can define $[\underline{E}, \phi]\cup [\underline{F}, \psi]$
			as in (\ref{eq:2.084}), by replacing $\ch$ by $\ch_{T_g}$,
			and the wedge product for differential forms by the product 
			defined as
			$
			(\phi\otimes \chi_V)\wedge (\psi\otimes \chi_W)=(\phi\wedge
			\psi) \otimes \chi_V\chi_W
			$
			for $\phi\otimes \chi_V, \psi\otimes \chi_W\in 
			\Omega^*(B^g,\R)\otimes R(T_g)$.
			Note that $\chi_V\chi_W=\chi_{V\otimes W}$.
			Following the same proof of \cite[Lemma 2.15]{LM20},
			we see that $\widehat{K}_{T_g}^0(B)$ is also a commutative 
			ring with unity $1=[\underline{\C}, 0]$ where $\underline{\C}$
			is equivariant trivial. Remark that in \cite[Definition 2.14]{LM20}, another equivariant differential K-ring
			$\widehat{K}^0_g(B)$ is defined. In fact, evaluating on $g\in G$, we can obtain a well-defined ring homomorphism
			\begin{align}\label{eq:5.07}
				\mathrm{ev}: \widehat{K}_{T_g}^0(B)\to \widehat{K}_{g}^0(B).  
			\end{align}
			Since $R(T_g)$ is an $R(G)$-module, 
			$\widehat{K}_{T_g}^0(B)$ is also an $R(G)$-module
			and this evaluation map is a homomorphism of $R(G)$-modules.
			
			The following proposition is the equivariant extension of 
			(\ref{d034}) and the analogue of \cite[Proposition 3.1]{LM20}.
			The proof of it is the same as that in \cite[Proposition 3.1]{LM20}.
			
			\begin{prop}
				For $g\in G$, there exists an exact sequence of $R(G)$-modules
				\begin{align}\label{eq:5.08}
					K_G^{1}(B)\overset{\footnotesize{\ch_{T_g
					}}}{\longrightarrow} 
					\left(\Omega^{\mathrm{odd}}(B^g, \R)/\Im\, d\right)
					\otimes R(T_g)
					\overset{\footnotesize{a}}{\longrightarrow} \widehat{K}_{T_g}^{0}(B) 
					\overset{\footnotesize{I}}{\longrightarrow} K_G^{0}(B)
					\longrightarrow 0,
				\end{align}	 
				where $K_G^*(B)$ denotes the equivariant topological $K$-group 
				of $B$, $a(\phi)=[\underline{0}, \phi]$, $I([\underline{E}, \phi])=[E]$ and $\ch_{T_g}$ on $K_G^1(B)$ is defined in the same way as \cite[(3.7)]{LM20} by replacing $\ch_g$ by $\ch_{T_g}$
				in (\ref{eq:4.4}) on equivariant geometric triples.
				
			\end{prop}
			
			Remark that in \cite[Proposition 3.1]{LM20}, the exactness
			was proved for the $S^1$-action. But the proof of it clearly applies to any compact Lie group action since by \cite[Proposition 2.4]{Segal68},
			for any compact Lie group $G$ and $G$-equivalent
			vector bundle $E$ over $B$, there exists 
			$G$-equivalent vector bundle $E^{\bot}$
			such that $E\oplus E^{\bot}$ is a $G$-equivariant trivial vector bundle.
			
			For a $G$-equivariant vector bundle $E$ over $B$, from \cite[Proposition 3.9]{Segal68}, $K_G^*(\mathbb{P}(E))$
			is generated as $K_G^*(B)$-algebra by the dual bundle 
			of the tautological bundle over $\mathbb{P}(E)$. Thus 
			the argument for Theorem \ref{prop:2.2} naturally
			can be extended to the 
			equivariant case. In other words, for the equivariant projection
			$\sigma:\mathbb{P}(E)\to B$, the pull-back map
			$\sigma^*: \widehat{K}_{T_g}^{0}(B)\to \widehat{K}_{T_g}^{0}(
			\mathbb{P}(E))$ is injective. So the splitting principle
			also holds for $T_g$-equivariant differential K-theory.
			
			In the following, we will construct a $\lambda$-ring structure
			on $\widehat{K}_{T_g}^{0}(B)$. Denote by 
			\begin{align}
				\Gamma_g(B):=\Gamma(B^g)\otimes R(T_g).
			\end{align}
			Then (\ref{eq:2.008}) for $\Gamma(B^g)$ and the multiplication
			in $R(T_g)$ induces the multiplication on $\Gamma_g(B)$.
			For $k\in \N$, we define the Adams operations
			\begin{align}
				\Psi^k_{\Gamma_g}:\Gamma_g(B)\to \Gamma_g(B), 
				\quad (\alpha\otimes \chi_{V_{\gamma_1}}, 
				\beta\otimes \chi_{V_{\gamma_2}})\mapsto
				(k^l\alpha\otimes \chi^k_{V_{\gamma_1}}, 
				k^l\beta\otimes \chi^k_{V_{\gamma_2}})
			\end{align}
			for $\alpha\in Z^{2l}(B^g, \R)$, $\beta\in \Omega^{2l-1}(B^g, 
			\R)/\mathrm{Im}\,d$ and irreducible representations
			$V_{\gamma_1}, V_{\gamma_2}$ of $T_g$. 
			Note that $\chi_{V_{\gamma}}^k=\chi_{V_{\gamma}^{\otimes k}}$
			and $V_{\gamma}^{\otimes k}$ is again an irreducible
			representation of $T_g$, thus
			 (\ref{eq:1.21}) holds for $\Psi^k_{\Gamma_g}$.
			By the same reason for
			Lemma \ref{lem:2.1}, from Theorem \ref{thm:1.07}, 
			$\Gamma_g(B)$ has a $\lambda$-ring structure
			$\lambda_{\Gamma_g}^n: \Gamma_g(B)\to \Gamma_g(B)$, $n\in \N$,
			constructed as in (\ref{eq:2.004}).
			
			For a cycle $(\underline{E}, \phi)$ of $\widehat{K}_{T_g}^0(B)$,
			comparing with (\ref{eq:2.043}), for $n\in \N$, set 
			\begin{align}\label{eq:5.12}
				\lambda^n(\underline{E},\phi):=\left(\Lambda^n(\underline{E}), 
				\left[\lambda_{\Gamma_g}^n(\ch_{T_g}(\underline{E}),\phi)\right]_{
					\mathrm{odd}}\right).
			\end{align}	
			where $[\cdot]_{\mathrm{odd}}$ is the component of $\Gamma_g(B)$
			in $\left(\Omega^{\mathrm{odd}}(B^g, \R)/\Im\, d\right)
			\otimes R(T_g)$.
			
			The proof of the following theorem is almost the same as that of 
			\cite[Theorem 2.6]{LM20} and Theorem \ref{thm:1.09}, just
			to extend everything to the equivariant version discussed above.
			
			\begin{thm}\label{thm:5.3}
				The map $\lambda^n$ in (\ref{eq:5.12}) induces a well-defined
				map
				\begin{align}
					\lambda^n: \widehat{K}_{T_g}^0(B)\to \widehat{K}_{T_g}^0(B)
				\end{align} 
				and the maps $\{\lambda^n; n\in \N \}$ is a $\lambda$-ring structure
				on $\widehat{K}_{T_g}^0(B)$.
			\end{thm}
			
Remark that it is difficult to construct the $\lambda$-ring structure,
even the pre-$\lambda$-structure on $\widehat{K}_g^0(B)$ in (\ref{eq:5.07}) although
it is the most suitable model to derive the localization
formula for
eta invariants in \cite{LM20}.

			By Theorem \ref{thm:5.3}, for $k\in \N^*$, the 
			Adams operations $\Psi^k: \widehat{K}_{T_g}^0(B)\to \widehat{K}_{T_g}^0(B)$ defined as in (\ref{eq:1.19})
			satisfy (\ref{eq:1.21}). 
			As in Proposition \ref{prop:3.1}, for $[\underline{E}, \phi]\in
			\widehat{K}_{T_g}^0(B)$, $\phi=\sum_{l\geq 1, \gamma}\alpha_l\otimes \chi_{V_{\gamma}}$ with $
			\alpha_l\in \Omega^{2l-1}(B^g, \R)/\Im\, d$ and $\chi_{V_{\gamma}}\in R(T_g)$ irreducible,
			\begin{align}\label{eq:5.03}
				\Psi^k([\underline{E}, \phi])
				=\left[ \Psi^k(\underline{E}),  \sum_{l\geq 1, \gamma}k^l\alpha_l\otimes 
				\chi^k_{V_{\gamma}} \right].
			\end{align}
			Moreover, the properties (\ref{eq:4.12}) and (\ref{eq:4.13})
			also hold in this equivariant case with respect to
			the curvature map
			$R: \widehat{K}^0_{T_g}(B)\to Z^{\mathrm{even}}(B^g)\otimes R(T_g)$ defined by $R([\underline{E}, \phi])=\ch_{T_g}(\underline{E})-d\phi$ and the map $I$ in 
			(\ref{eq:5.08}).


			\
			



\begin{thebibliography}{10}
				
				\bibitem{A62}
				J.~F. Adams,
				\emph{Vector fields on spheres},
				Ann. of Math. (2) \textbf{75} (1962), 603--632.
				
				\bibitem{A67}
				M.~F. Atiyah,
				\emph{K-theory},
				Lecture notes by D.~W.~Anderson, W.~A.~Benjamin, Inc., New York-Amsterdam, 1967, v+166+xlix pp.
				
				
				\bibitem{ATall69}
				M.~F. Atiyah and D.~O. Tall, 
				\emph{Group representations, {$\lambda $}-rings 
					and the {$J$}-homomorphism}, 
				Topology \textbf{8} (1969), 253--297.
				
				
				
				\bibitem{BerthelotTh}
				P.~Berthelot, A.~Grothendieck, and L.~Illusie, \emph{Th\'eorie des
					{I}ntersections et {T}h\'eor\'eme de {R}iemann-{R}och (SGA6)}, 
				Lecture Notes in Mathematics,, vol. 225, 
				Springer-Verlag, Berlin, 1971.
				
				\bibitem{Bunke10}
				U.~Bunke, 
				\emph{Adams operations in smooth K-theory},
				Geom. Topol. \textbf{14} (2010), no.~4, 2349--2381.
				
				\bibitem{BS09}
				U.~Bunke and T.~Schick,
				\emph{Smooth K-theory},
				Ast\'erisque \textbf{328} (2009), 45--135.
				
				\bibitem{FH00}
				D.~S. Freed and M.~Hopkins, 
				\emph{On {R}amond-{R}amond fields and {$K$}-theory},
				J. High Energy Phys., (2000) no.~5 , Paper 44, 14 pp.
				
				
				\bibitem{FreedLott10}
				D.~S. Freed and J.~Lott, 
				\emph{An index theorem in differential {$K$}-theory},
				Geom. Topol. \textbf{14} (2010), no.~2, 903--966. 
				
				
				\bibitem{GS90c}
				H.~Gillet and C.~Soul\'e, 
				\emph{Characteristic classes for algebraic vector
					bundles with {H}ermitian metric. {II}},
				Ann. of Math. (2) \textbf{131}
				(1990), no.~1, 205--238.
				
				\bibitem{Grothendieck} 
				A.~Grothendieck,
				\emph{La th\'eorie des classes de Chern } 
				Bulletin de la Soci\'et\'e Math\'ematique de France, \textbf{86} (1958), 137--154.
				
				\bibitem{HS05}
				M.~J.~Hopkins and I.~M.~Singer,
				\emph{Quadratic functions in geometry, topology, and M-theory},
				J. Differential Geom. \textbf{70} (2005), no.~3, 329--452.
				
				\bibitem{K08}
				K.~R. Klonoff, 
				\emph{An index theorem in differential K-theory},
				Ph.D. thesis, The University of Texas at Austin, 2008,
				119~pp., ISBN:~978-0549-70973-2, ProQuest LLC.
				
				
				\bibitem{LNM308}
				D.~Knutson,
				\emph{{$\lambda$}-rings and the representation theory of the 
					symmetric group},
				Lecture Notes in Mathematics, Vol. \textbf{308}. Springer-Verlag, 
				Berlin-New York, 1973. iv+203 pp.
				
				
				
				
				
				
				
				
				
				\bibitem{LM20}
				B.~Liu and X.~Ma,
				\emph{Differential {$K$}-theory and localization formula for 
					$\eta$-invariants},
				Invent. Math. \textbf{222} (2020), no. 2, 545--613.
				
				\bibitem{Ro01} D.~Roessler, 
				\emph{Lambda-structure on {G}rothendieck groups of {H}ermitian 
					vector bundles}, 
				Israel J. Math. \textbf{122} (2001), 279--304.
				
				\bibitem{MM07}
				X.~Ma and G.~Marinescu, 
				\emph{Holomorphic {M}orse inequalities and {B}ergman
					kernels}, Progress in Mathematics, vol. 254, Birkh\"auser Verlag, Basel,
				2007. 
				
				\bibitem{Segal68}
				G.~Segal, 
				\emph{Equivariant {$K$}-theory}, Inst. Hautes \'Etudes Sci. Publ.
				Math. (1968), no.~34, 129--151.
				
				\bibitem{SS10}
				J.~Simons and D.~Sullivan,
				\emph{Structured vector bundles define differential K-theory},
				in \emph{Quanta of maths}, 
				Clay Math. Proc. \textbf{11}, Amer. Math. Soc., Providence, RI, 2010, 579--599.
				
				\bibitem{Z01}	W.~Zhang, 
				\emph{Lectures on {C}hern-{W}eil theory and {W}itten deformations},
				Nankai Tracts in Mathematics, vol.~4,
				World Scientific Publishing Co., Inc., River Edge, NJ, 2001.
				
			\end{thebibliography}

		\end{document}